\documentclass[12pt]{article}\textwidth 160mm\textheight 235mm
\usepackage{amsmath}
\usepackage{amsfonts,amsthm,latexsym,epsf}
\usepackage{amssymb}
\usepackage{graphicx}
\usepackage{graphics}
\usepackage{color}

\def\sce{\setcounter{equation}{0}}

\def\dist{\operatornamewithlimits{\mbox{dist}}}
\def\half{{\frac{1}{2}}}
\def\argmin{\operatornamewithlimits{\mbox{argmin}}}

\def\th{{\theta}}
\usepackage{color}
\def\tto{\;{\lower 1pt \hbox{$\rightarrow$}}\kern -10pt
\hbox{\raise 2pt \hbox{$\rightarrow$}}\;}
\def\Hat{\widehat}

\def\Bar{\overline}
\def\eps{\epsilon}
\def\lam{\lam}
\def\ra{\rangle}
\def\la{\langle}
\def\ve{\varepsilon}
\def\B{I\!\!B}

\def\R{I\!\!R}

\def\ox{\bar{x}}
\def\oy{\bar{y}}
\def\oz{\bar{z}}
\def\ov{\bar{v}}

\def\ou{\bar{u}}

\def\op{\bar{p}}

\def\Hat{\widehat}
\def\hat{\Hat}

\def\gph{\mbox{\rm gph}\,}

\def\cl*co{\mbox{\rm cl}^*\mbox{\rm co}\,}
\def\cl{\mbox{\rm cl}\,}

\def\cl{\mbox{\rm cl}\,}

\def\dn{\downarrow}
\def\O{\Omega}
\def\ph{\varphi}
\def\emp{\emptyset}
\def\st{\stackrel}

\def\lm{\lambda}
\def\gg{\gamma}
\def\dd{\delta}

\def\Th{\Theta}

\def\hs7{\hspace*{7pt}}
\def\th{\theta}
\def\kk{\kappa}

\def\R{\mathbb{R}}
\def\rr{\mathbb{R}}
\def\lin{\mathop{\rm lin}}

\newtheorem{Theorem}{Theorem}[section]
\newtheorem{Proposition}[Theorem]{Proposition}
\newtheorem{Remark}[Theorem]{Remark}
\newtheorem{Lemma}[Theorem]{Lemma}

\newtheorem{Definition}[Theorem]{Definition}
\newtheorem{Example}[Theorem]{Example}
\begin{document}
\begin{center}
{\Large\textbf{Graphical Derivatives and Stability Analysis for Parameterized Equilibria with Conic Constraints}}\\[3ex]
B. S. MORDUKHOVICH\footnote{Department of Mathematics, Wayne State University, Detroit, Michigan, USA(boris@math.wayne.edu). Research of this author was partly supported by the National Science Foundation under grant DMS-1007132.}, J. V. OUTRATA\footnote{Institute of Information Theory and Automation, Czech Academy of Science, Prague, Czech Republic (outrata@utia.cas.cz) and Centre for Informatics and Applied Optimization, Federation University of Australia, Ballarat, Australia. Research of this author was partly supported by grant P201/12/0671 of the Grant Agency of the Czech Republic and the Australian Research Council under grant DP-110102011.} and H. RAM\'{I}REZ C.\footnote{Departamento de Ingenier\'{i}a Matem\'{a}tica y Centro de Modelamiento Matem\'{a}tico, Universidad de Chile, Santiago, Chile (hramirez@dim.uchile.cl). Research of this author was partly supported by FONDECYT Project~1110888 and BASAL Project Centro de Modelamiento Matem\'{a}tico, Universidad de Chile.}\\[2ex]
{\em Dedicated to Lionel Thibault in honor of his 65th birthday}\\[2ex]
\end{center}\small
\textbf{Abstract:} The paper concerns parameterized equilibria governed by generalized equations whose multivalued parts are modeled via regular
normals to nonconvex conic constraints. Our main goal is to derive a precise pointwise second-order formula for calculating the graphical derivative of the solution maps to such generalized equations that involves Lagrange multipliers of the corresponding KKT systems and critical cone directions. Then we apply the obtained formula to characterizing a Lipschitzian stability notion for the solution maps that is known as isolated calmness.\\[1ex]

\section{Introduction}\sce

This paper pursues a twofold goal. The main attention is paid to developing generalized differential calculus of variational analysis to which Lionel Thibault made crucial, pioneering contributions. These aspects of our present study, being certainly of their own interest, are motivated by the subsequent application to characterizing the so-called isolated calmness property of stability analysis for parameterized equilibria represented as the solution map to the generalized equation (GE)
\begin{equation}\label{GE}
0\in f(x,y)+\widehat{N}_{\Gamma}(y)\;\mbox{ with }\;\Gamma:=g^{-1}(\Theta),
\end{equation}
which contains the regular normal cone $\widehat{N}_{\Gamma}$ (see Section~2 for this and the other major constructions of generalized differentiation employed in the paper) to the given, usually nonconvex set $\Gamma$. By the general results of variational analysis (see Section~5), achieving the latter goal requires the usage of the graphical derivative of the solution map $S:x\mapsto y$ defined by
\begin{equation}\label{sol}
S(x):=\big\{y\in\mathbb{R}^{m}\big|\;0\in f(x,y)+\widehat{N}_{\Gamma}(y)\big\}
\end{equation}
and its calculation in terms of the initial problem data of \eqref{GE} and the associated values computed at the reference solution point. This amounts to developing a calculus rule for the expression of the graphical derivative of the normal cone mapping $\hat{N}_{\Gamma}(\cdot)$.

It has been well recognized in variational analysis that developing calculus rules (even of the inclusion type) for nonrobust, tangentially generated graphical derivatives is a challenging issue. In fact, not much has been known in this direction; see, e.g., \cite{RW}. This significantly distinguishes tangentially generated derivative constructions from limiting normals and normally generated coderivatives, which--despite their intrinsic nonconvexity--enjoy comprehensive calculus rules based on variational/extremal principles of variational analysis; see the books \cite{m06,RW} and the references therein.

In this paper we focus on the special class of set-valued mappings/multifunctions $S\colon\R^n\tto\R^m$ given in \eqref{sol} and observe that such mappings accumulate certain first-order information about optimization and equilibrium problems via the regular  normal cone $\Hat N_\Gamma$ to the constraint set $\Gamma$. Therefore, generalized differentiation of $S$ leads us to a second-order object, and the desired formula for the graphical derivative of this multifunction can be treated as a result of {\em second-order calculus}.

Some results on generalized differentiation of set-valued mappings of type \eqref{sol} are available in the literature. Namely, the paper \cite{OR} contains the calculation of the limiting coderivative of the solution map to a counterpart of GE \eqref{GE}, where $\Theta$ is a Carthesian product of the Lorentz cones. Our recent paper \cite{MOR} provides a precise second-order formula to calculate the regular coderivative of the solution map $S$ given in \eqref{sol} under natural assumptions. Furthermore, the same paper \cite{MOR} contains a formula for calculating the graphical derivative of \eqref{sol} but only under the convexity assumption on $\Gamma$, which is rather restrictive, being however unavoidable in the technique of \cite{MOR}. Observe also that the convexity assumption on $\Gamma$ is not imposed in \cite{Kr} while the set $\Theta$ in \eqref{sol} is assumed to be a convex polyhedron. This excludes from consideration many important classes in conic programming, e.g., second-order cone programs and semidefinite programs, which are among the main motivations for our current research.

In this paper we are able to completely avoid the convexity assumptions on $\Gamma$ and significantly relax the polyhedrality assumption on $\Theta$. The key ingredients allowing us to proceed in this way are the usage of the recent characterizations of {\em full stability} of local minimizers in problems of conic programming \cite{mnr} and the projection representation for nonconvex {\em prox-regular} sets taken from \cite{PRT}. Furthermore, an important role in our device is played by a new local geometric condition on the underlying set $\Th$ in the conic constraint $g(y)\in\Th$, which is labeled as the {\em projection derivation condition} (PDC) and which holds under the (second-order) extended polyhedricity condition from \cite{BSbook} and therefore also under  the stronger polyhedricity and polyhedrality properties of convex sets.\vspace*{0.05in}

The rest of the paper is organized as follows. In Section~2 we state the problem, introduce and discuss the standing assumptions, and recall the notions of first-order and second-order generalized differentiation widely used in the formulations and proofs of the subsequent results in the paper.

Section~3 is mainly devoted to the new results on the directional differentiability of the projection operator $P_\Gamma$ associated with the constraint set $\Gamma$ in \eqref{GE}. We prove here the directional differentiability of $P_\Gamma$ and establish a precise representation of the directional derivative $P'_\Gamma(u;h)$ via the directional derivative $P'_\Th$ without imposing the convexity assumption on $\Gamma$ and/or the projection derivation condition (and hence any polyhedricity-like assumption) on $\Th$. The aforementioned characterizations of full stability in conic programming play a crucial role in this section.

In Section~4 we formally introduce and discuss the aforementioned PDC property for $\Th$ that is crucial for the subsequent calculation of the graphical derivative of the solution map and its application to isolated calmness. In particular, relationships between the new PDC and the polyhedricity and extended polyhedricity conditions on $\Th$ are established and illustrated in this section.

Section~5 contains the main results of the paper providing second-order formulas for calculating the graphical derivative of the regular normal cone mapping $\Hat N_\Gamma$  and then of the solution map $S$ from \eqref{sol} in terms of Lagrange multipliers of the perturbed KKT system and the critical cone of $\Th$ under the projection derivation condition imposed on $\Th$ at the reference solution point.

Section~6 is devoted to the application of the graphical derivative  formulas and other calculus results to deriving sufficient conditions as well as complete characterizations of the isolated calmness property of $S$ at $(\ox,\oy)$ in terms of the problem data. We illustrate the efficient usage of these conditions in the case of equilibrium systems governed by the nonpolyhedral second-order (Lorentz) cone in $\R^3$. In the concluding Section~7 we discuss some perspective topics for future research.\vspace*{0.05in}

Our notation is standard throughout the whole paper, except from special symbols defined in the places where they first appear. Recall that $\mathbb{R}^{n}$ is the $n$-dimensional Euclidean space, $I$ is the identity matrix, $\O^{\perp}$ signifies the orthogonal complement to the set $\O$, and $A^T$ stands for the matrix or vector transposition. We denote by $F\colon\R^n\tto\R^m$ a set-valued mapping, which takes values in the subsets of $\R^m$. This distinguishes set-valued mappings from vector-valued ones denoted by $f\colon\R^n\to\R^m$. In the latter case, the symbol $f^{\prime}(x;h)$ stands for the classical directional derivative of $f$ at the point $x\in\R^n$ in the direction $h\in\R^n$. As usual, $\B(x;r)$ denotes the closed ball centered at $x$ with radius $r>0$ while $\B$ signifies the closed unit ball of the space in question.

\section{Problem Formulation and Preliminaries}\sce

The major object of our analysis is the parameter-dependent {\em generalized equation}
\begin{equation} \label{eq-1}
0\in f(x,y)+\widehat{N}_{\Gamma}(y)
\end{equation}
in Robinson's formalism \cite{rob}, which has been well recognized as a convenient model to study various problems of optimization and equilibria. In \eqref{eq-1} we have: $x\in\mathbb{R}^{n}$ is the {\em parameter}, $y\in\mathbb{R}^{m}$ is the {\em decision variable}, the mapping $f:\mathbb{R}^{n}\times\mathbb{R}^{m}\rightarrow\mathbb{R}^{m}$ is continuously differentiable, and $\widehat{N}_{\Gamma}(y)$ stands for the (Fr\'echet) {\em regular normal cone} to the set $\Gamma\subset\R^m$ at the point $y\in\Gamma$ defined by
\begin{eqnarray}\label{rnc}
\Hat N_\Gamma(y):=\Big\{v\in\R^m\Big|\;\limsup_{u\st{\Gamma}{\to}y}\frac{\la v,u-y\ra}{\|u-y\|}\le 0\Big\},
\end{eqnarray}
where the symbol $u\st{\Gamma}{\to}y$ indicates that $u\to y$ with $u\in\Gamma$. In what follows we address the GE model \eqref{eq-1} with $\Gamma$ described by the {\em conic constraint}
\begin{equation}\label{eq-2}
\Gamma=g^{-1}(\Theta)\Longleftrightarrow g(y)\in\Th,
\end{equation}
where $g:\mathbb{R}^{m}\rightarrow\mathbb{R}^{l}$ is twice continuously differentiable and $\Theta\subset\mathbb{R}^{l}$ is a closed convex cone. We associate with (\ref{eq-1}) the parameter-dependent {\em solution map} $S:\mathbb{R}^{n}\rightrightarrows\mathbb{R}^{m}$ defined by
\begin{equation}\label{eq:solution-map}
S(x):=\big\{y\in\mathbb{R}^{m}\big|\;0\in f(x,y)+\widehat{N}_{\Gamma}(y)\big\},\quad x\in\mathbb{R}^{n}.
\end{equation}
As mentioned in Section~1, the twofold goal of this paper is to derive a verifiable formula for calculating the graphical derivative of the solution map  $S$ from \eqref{eq:solution-map} and apply it to characterizing the isolated calmness property of $S$ at the reference point of its graph.

Given an arbitrary set-valued mapping $F\colon\R^n\tto\R^m$ and the point $(\ox,\oy)$ from its graph
$$
\gph F:=\big\{(x,y)\in\R^n\times\R^m\big|\;y\in F(x)\big\},
$$
the {\em graphical derivative} of $F$ at $(\ox,\oy)$ is the mapping $DF(\ox,\oy)\colon\R^n\tto\R^m$ defined by
\begin{eqnarray}\label{der}
D F(\ox,\oy)(u):=\big\{v\in\R^m\big|\;(u,v)\in T_{{\rm\small gph}\,F}(\ox,\oy)\big\},\quad u\in\R^n,
\end{eqnarray}
where the {\em tangent/contingent cone} to a set $\O\subset\R^s$ at a point $\oz\in\O$ is given by
\begin{equation}\label{tan}
T_\Omega(\bar z):=\big\{z\in\R^s\big|\;\exists\,t_k\downarrow 0,\;z_k\to z\;\mbox{ as }\;k\to\infty\;\mbox{ with }\;\bar z+t_k z_k\in\Omega\big\}.
\end{equation}
We refer the reader to \cite{RW} for more information on these constructions. Let us mention here that the (convex) regular normal cone \eqref{rnc} to $\O$ at $\oz$ is dual/polar to the tangent cone \eqref{tan}, i.e.,
\begin{eqnarray*}
\Hat N_{\Omega}(\bar z)=T_{\Omega}(\bar y)^{*}:=\big\{w\in\R^s\big|\;\langle w,z\rangle \le 0\;\mbox{ for all }\;z\in T_{\Omega}(\bar z)\big\}
\end{eqnarray*}
while not vice versa, since the tangent cone \eqref{tan} is generally nonconvex.\vspace*{0.05in}

Next we formulate our standing assumptions in this paper, which are standard in conic programming; see, e.g., the book \cite{BSbook} and the references therein.\\

{\bf Standing assumptions:}

\begin{description}
\item {\bf (A1)} The set $\Theta$ is ${\cal C}^{2}$-{\em reducible} to a closed convex set $\Xi\subset \mathbb{R}^{q}$ at $\bar z:=g(\bar{y})$, and
the reduction is {\em pointed}. This means that there exist a neighborhood $V$ of $\bar z$ and a ${\cal C}^{2}$-smooth mapping $h:V\to\mathbb{R}^{q}$ such that: {\bf(i)} for all $z\in V$ we have $z\in\Theta$ if and only if $h(z)\in\Xi$, where the cone $T_\Xi(h(\bar{z}))$ is pointed; {\bf(ii)} $h(\bar z)=0$ and the derivative mapping $\nabla h(\bar z):\mathbb{R}^{l}\to\mathbb{R}^{q}$ is {\em surjective/onto}, i.e., the Jacobian matrix $\nabla h(\bar z)$ has full rank.

\item {\bf (A2)} The point $\bar{y}\in\mathbb R^m$ is {\em nondegenerate} for $g$ with respect to $\Theta$, i.e.,
\begin{eqnarray*}
\nabla g(\bar y)\R^m+\lin\big(T_\Theta(\bar z)\big)=\R^l,
\end{eqnarray*}
where $\lin(Q)$ denotes the largest linear subspace of $\mathbb R^l$ contained in $Q\subset\mathbb R^l$.

\item {\bf (A3)} The metric projection operator onto $\Theta$, denoted by $P_{\Theta}$, is {\em directionally differentiable} on $\R^l$.
\end{description}

It occurs that assumption {\bf (A3)} holds automatically for a large class of sets typically encountered in conic programming. To describe such sets $\O\subset\R^s$, fix $\bar z\in\Omega$ and $h\in T_\Omega(\bar z)$. Recall that
\begin{eqnarray}\label{2tan}
T^{2}_{\Omega}(\bar z,h):=\Big\{w\in\;\R^s\Big|\;\dist\Big(\bar z+th+\half t^{2}w;\Omega\Big)=o(t^{2})\;\mbox{ for all }\;t>0\Big\}
\end{eqnarray}
is known as the (inner) {\em second-order tangent set} of $\Omega$ at $\bar z$ in the direction $h$. According to \cite{BCS98}, the set $\O$ is {\em second-order regular at} $\oz\in\O$ if for every sequence $z_k\to\oz$ in the form $z_k=\oz+t_kh+\frac{1}{2}t^2_k r_k$ with $t_k\dn 0$ and $t_k r_k\to 0$ as $k\to\infty$ it follows that
\begin{eqnarray}\label{2reg}
\lim_{k\to\infty}\big[\dist\big(r_k;T^2_\O(\ox,h)\big)\big]=0.
\end{eqnarray}
The set $\O$ is said to be {\em second-order regular} if it is second-order regular at every point $\oz\in\O$.

We refer the reader to \cite{BCS98}, \cite[Section~3.3.3]{BSbook}, and the recent paper \cite{sh} for various useful properties of second-order regular sets, which cover a large territory in second-order variational analysis and optimization. In particular, if $\O$ is second-order regular at $\oz$, then the inner second-order tangent set \eqref{2tan} agrees with its outer counterpart (which is not employed in this paper) and also $T^2_\O(\oz;h)\ne\emp$ for any $h\in T_\O(\oz)$. Among sufficient conditions for second-order regularity we mention the validity of this property at $\oz\in\O$ for any convex set $\O$ that is {\em cone reducible} at $\oz$, i.e., the set $\Xi$ in {\bf(A1)} is a pointed cone. The latter property holds, at any $\oz\in\O$, for many important classes of sets in conic programming, e.g., for convex polyhedra, for the cone of symmetric positive-semidefinite matrices in semidefinite programming, and for the {\em second-order/Lorentz/ice-cream cone} given by
\begin{eqnarray}\label{ic}
{Q}^l:=\big\{(\th_1,\ldots,\th_l)\in\R^l\big|\;\th_1\ge\|(\th_2,\ldots,\th_l)\|\big\}
\end{eqnarray}
with the Euclidean norm $\|\cdot\|$ that describes problems of second-order cone programming.

The principal result of \cite[Theorem~7.2]{BCS98} says the following: Given a closed and convex set $\O\subset\R^s$, its single-valued metric projection $P_\O\colon\R^s\to\O$, and $\oz=P_\O(\oy)$ with $\oy\in\R^s$, the second-order regularity of $\O$ at $\oz$ ensures the directional differentiability of $P_\O$ at $\oy$. This shows that  assumption {\bf (A3)} holds automatically for second-order regular sets $\Th$ in the conic constraint \eqref{eq-2}.

Finally in this section, we recall the notions of the (Mordukhovich) limiting normal cone and coderivative used in the proofs of our main results; see \cite{m06,RW} for more details and references on these constructions. Given a set $\O\subset\R^s$, the {\em limiting normal cone} to $\O$ at $\oz\in\O$ is defined by
\begin{eqnarray}\label{nc}
N_\O(\oz):=\big\{v\in\R^s\big|\,\exists z_k\to\oz,\;v_k\to v\;\mbox{ with }\;z_k\in\O,\;v_k\in\Hat N_\O(z_k)\big\}.
\end{eqnarray}
Given a mapping $F\colon\R^n\tto\R^m$ and a point $(\ox,\oy)\in\gph F$, the {\em limiting coderivative} of $F$ at $(\ox,\oy)$ is the set-valued mapping $D^*F(\ox,\oy)\colon\R^m\tto\R^n$ defined by using the normal cone \eqref{nc} as
\begin{eqnarray}\label{cod}
D^*F(\ox,\oy)(v):=\big\{u\in\R^n\big|\;(u,-v)\in N_{{\rm\small gph}\,F}(\ox,\oy)\big\},\quad v\in\R^m.
\end{eqnarray}
These constructions and their second-order combinations allow us to characterize the fundamental notion of full stability of local minimizers in conic programs employed in what follows.

\section{Directional Derivatives of Projection Operators}\sce

The main goal of this section is to establish relationships between the directional derivatives of the projection operator $P_\Gamma$ onto the conic constraint set $\Gamma$ from \eqref{eq-2} and the projection operator $P_\Th$ onto the underlying cone $\Th$. To proceed, consider first the auxiliary {\em linear GE}
\begin{equation} \label{eq-100}
0\in y-u+\widehat{N}_{\Gamma}(y),\quad y\in\Gamma,\;u\in\R^m,
\end{equation}
and associate with (\ref{eq-100}) the canonically perturbed {\em Karush-Kuhn-Tucker} (KKT) system
\begin{equation}\label{eq-101}
\begin{array}{ll}
u=y+\big(\nabla g(y)\big)^{T}\nu\\
s\in-g(y)+N_{\Theta^{*}}(\nu),
\end{array}
\end{equation}
where $\nu\in\R^l$ is the corresponding {\em Lagrange multiplier}. Denote by ${\cal T}$ the mapping $(u,s)\mapsto(y,\nu)$ defined by (\ref{eq-101}) and pick a  vector $\bar{u}$ such that $\bar{y}\in P_{\Gamma}(\bar{u})$. Note that under the posed standing assumptions there  is a unique Lagrange multiplier $\bar{\nu}\in\R_l$ such that
\begin{equation}\label{eq-99}
(\bar{y},\bar{\nu})\in{\cal T}(\bar{u},0).
\end{equation}

The next proposition, which is of its own interest, plays an important role in deriving the major results of this paper. Its proof is based on the recent second-order characterizations of the fundamental notion of full stability in optimization introduced in \cite{lpr}.

Recall this notion adapted to the case of conic programs considered in what follows:
\begin{eqnarray}\label{cp}
\mbox{minimize }\;\ph(y,\op)\;\mbox{ subject to }\;q(y,\op)\in\Th,
\end{eqnarray}
where the cost function $\ph\colon\R^m\times\R^d\to\R$ and the constraint mapping $q\colon\R^m\times\R^d\to\R^l$ are ${\cal C}^2$-smooth  around the reference pair of $(\oy,\op)$ of the solution vector $y\in\R^m$ and the nominal value of the basic parameter $p\in\R^d$. Consider now the perturbed version ${\cal P}(u,p)$ of \eqref{cp} involving also another (tilt) parameter $u\in\R^m$ and given in the form:
\begin{eqnarray}\label{per-cp}
\mbox{minimize }\;\psi(y,p)-\la u,y\ra\;\mbox{ with }\;\psi(y,p):=\ph(y,u)+\dd_\Th\big(q(y,p)\big),\;(y,p)\in\R^m\times\R^d,
\end{eqnarray}
where $\dd_\Th$ stands for the indicator function of the set $\Th$. Fix $\gg>0$ and define the local value function and solution map for the parametric problem ${\cal P}(u,p)$ in \eqref{per-cp} by, respectively,
\begin{eqnarray*}
\begin{array}{ll}
m_\gg(u,p):=\inf\big\{\psi(y,p)-\la u,y\ra\big|\;\|y-\oy\|\le\gg\big\},\\\\
M_\gg(u,p):=\mbox{argmin}\big\{\psi(y,p)-\la u,y\ra\big|\;\|y-\oy\|\le\gg\big\}.
\end{array}
\end{eqnarray*}
We say that $\ox$ is a (Lipschitzian) {\em fully stable local minimizer} of ${\cal P}(\ou,\op)$ if there exist positive numbers $\gg,\kk$ and a neighborhood $U\times V$ of $(\ou,\op)$ such that the mapping $(u,p)\mapsto M_\gg(u,p)$ is single-valued on $U\times V$ with $M_\gg(\ou,\op)=\oy$ satisfying the Lipschitz condition
$$
\|M_\gg(u_1,p_1)-M_\gg(u_2,p_2)\|\le\kk\big(\|u_1-u_2\|+\|p_1-p_2\|\big)\;\mbox{ for all }\;u_1,u_2\in U,\;p_1.p_2\in V
$$
and the value function $(u,p)\mapsto m_\gg(u,p)$ is also Lipschitz continuous around $(\ou,\op)$.

This notion has been recognized as an important stability concept in optimization and has been completely characterized via various second-order conditions. We refer the reader to \cite{lpr} and the recent papers \cite{mn,mnr,mos,Full,ms} for such characterizations and their applications to broad classes of optimization and control problems.\vspace*{0.05in}

Now we are ready to formulate and prove the aforementioned proposition important in what follows. The second-order condition in its first part is expressed in terms of the coderivative \eqref{cod} of the normal cone mapping $N_\Th$ generated by the cone $\Th$ from the conic constraint \eqref{eq-2}.

\begin{Proposition}[\bf single-valued Lipschitzian localization of the KKT system]\label{prop1} Consider the triple $(\bar{u},\bar{y},\bar{\nu})$ satisfying \eqref{eq-99} via the KKT system \eqref{eq-101}. The following assertions hold:

{\bf (i)} Assume that for all $w\in\mathbb{R}^{m}\setminus\{0\}$ we have the second-order condition
\begin{eqnarray}\label{eq-102}
\left\langle w,(I+\sum^{l}_{i=1}\bar{\nu_{i}}\nabla^{2}g_{i}(\bar{y}))w\right\rangle+\big\langle\nabla g(\bar{y})w, D^{*}N_{\Theta}(g(\bar{y}),\bar{\nu})(\nabla g(\bar{y})w)\big\rangle>0.
\end{eqnarray}
Then the set-valued mapping ${\cal T}$ from \eqref{eq-99} admits a single-valued and Lipschitz continuous localization around the quadruple $(\bar{u},0,\bar{y},\bar{\nu})$.

{\bf (ii)} If $\ou=\oy$ in \eqref{eq-99}, then the conclusion in {\bf (i)} is valid without assuming \eqref{eq-102}.
\end{Proposition}
\proof To verify the conclusion of {\bf(i)}, which means Robinson's {\em strong regularity} \cite{Ro} of the generalized equation corresponding to (\ref{eq-101}),  we employ \cite[Theorem~5.6]{mnr} and deduce from the equivalence {\bf (i)}$\Longleftrightarrow${\bf (iii)} therein that, under our standing assumptions, the conclusion in {\bf (i)} amounts to saying that $\oy$ is a {\em fully stable} local minimizer corresponding to $(\ou,0)$ of the problem ${\cal P}(u,s)$ given by
\begin{eqnarray}\label{per-cp1}
\mbox{minimize }\;\frac{1}{2}\|y\|^2-\la u,y\ra\;\mbox{ subject to }\;g(y)+s\in\Th,
\end{eqnarray}
which is a specification of \eqref{per-cp} with $\ph(y,u)=\frac{1}{2}\|y\|^2$, $p=s\in\R^l$, and $q(y,p)=g(y)+s$. Then \cite[Theorem~5.6(iv)]{mnr} tells us that condition \eqref{eq-102} is a characterization of full stability of $\oy$ in the problem ${\cal P}(\ou,0)$ from \eqref{per-cp1} with the Lagrange multiplier $\bar\nu$. This verifies assertion {\bf (i)}.

To justify assertion {\bf (ii)}, it suffices to show that condition \eqref{eq-102} holds automatically if $\ou=\oy$. Indeed, condition \eqref{eq-102} means that for any $v\in D^*N_\Th(g(\oy),\bar\nu)(\nabla g(\oy)w)$ with $w\ne 0$ we have
\begin{eqnarray}\label{eq-102a}
\|w\|^2+\Big\la w,\sum_{i=1}^l\bar\nu_i\nabla^2g_i(\oy)w\Big\ra+\big\la\nabla g(\oy)w,v\big\ra>0.
\end{eqnarray}
Since $\ou=\oy$, it follows from \eqref{eq-101} that $\bar\nu=0$, and thus the middle term in \eqref{eq-102a} disappears. Furthermore, the maximal monotonicity of the normal cone mapping implies by \cite[Theorem~2.1]{PR} that $\la\nabla g(\oy)w,v\ra\ge 0$. This shows that \eqref{eq-102a} holds, which completes the proof of the proposition.\endproof

\begin{Remark}[\bf on the second-order condition]\label{rem1} {\rm Condition \eqref{eq-102} can be treated as a proper extension of the classical {\em strong second-order sufficient condition} \cite{Ro} to which \eqref{eq-102} reduces in the case of $\Th=\R^l_-$, i.e., in the case of standard equality and inequality constraints as in nonlinear programming. We refer the reader to \cite{mn,mnr,mos,Full,ms} for constructive versions of \eqref{eq-102} in other constraint systems. Note that \eqref{eq-102} is satisfied when $g$ is {\em $\Th$-convex}, i.e., the set
$$
\big\{(y,z)\in\R^m\times\R^l\big|\;g(y)-z\in\Th\big\}
$$
is convex. Indeed, the latter property is equivalent for ${\cal C}^2$-smooth mappings $g$ to the condition
$$
\big\la\nabla^2g(y)(h,h),\nu\big\ra\ge 0\;\mbox{ for all }\;\nu\in\Th^*\;\mbox{ and }\;y,h\in\R^m.
$$
Therefore the $\Th$-convexity of $g$ ensures that the matrix $\sum^{l}_{i=1}\bar\nu_{i}\nabla^2g_{i}(\bar{y})$ is positive semidefinite, and so \eqref{eq-102} holds by the discussion above; cf.\ also \cite{MOR}. On the other hand, it is not hard to construct simple examples showing that the violation of the second-order condition \eqref{eq-102} for nonconvex sets $\Gamma$ prevents the validity of the conclusion in Proposition~\ref{prop1}(i).}
\end{Remark}

However, it is rather surprising to see that we {\em do not need} to assume the second-order condition \eqref{eq-102} in the rest of the paper. In particular, justifying the main results of the paper presented below requires only the usage  of assertion (ii) in Proposition~\ref{prop1} without imposing any convexity assumption on the set $\Gamma$ and therefore the $\Th$-convexity of the mapping $g$ as in \cite{MOR}.

To proceed, we first recall the notion of prox-regularity \cite{RW} for closed sets in finite dimensions and an important property of projections onto such sets. A set $\O\subset\R^s$ is {\em prox-regular} at $\oz\in\O$ for $\ov\in N_\O(\oz)$ if $\O$ is locally closed around $\oz$ and there are numbers $\ve>0$ and $\rho\ge 0$ such that
\begin{eqnarray*}
\la v,u-z\ra\le\frac{1}{2}\|u-z\|^2\;\mbox{ for all }\;u\in\O\cap{\rm\mathbb{B}}(\oz;\ve),\,v\in N_\O(u),\,\|v-\ov\|<\ve,\,\|z-\oz\|<\ve.
\end{eqnarray*}
The set $\O$ is called prox-regular at $\oz$ if it is prox-regular at $\oz$ for any $\ov\in N_\O(\oz)$.

Besides the validity of this property for closed convex sets, it holds for a large variety of other ``nice" sets broadly encountered in problems of variational analysis and optimization; see. e.g., \cite{RW} and the references therein. In particular, the conic constraint set $\Gamma$ under consideration in this paper \eqref{eq-2} is prox-regular at $\oy\in\Gamma$ (even better--``strongly amenable") under the nondegeneracy assumption {\bf (A2)}; see \cite[Proposition~13.32 and Exercise~10.25]{RW}.

Now we invoke the following result borrowed from \cite{PRT}, which holds for our underlying set $\Gamma$.

\begin{Proposition}[\bf projection representation for prox-regular sets]\label{thm1}  Given $\O\subset\R^s$ and $r>0$, consider the truncation of the normal cone
\begin{eqnarray*}
N^{r}_{\O}(z):=\left\{\begin{array}{ll}
N_{\O}(z)\cap\mathbb{B}(z;r)&\mbox{if }\;z\in\O,\\
\emp&\mbox{otherwise.}
\end{array}\right.
\end{eqnarray*}
Assume that $\O$ is prox-regular at $\oz\in\O$. Then there exists a neighborhood ${\cal O}$ of $\oz$ on which the projection operator $P_\O$ is single-valued and Lipschitz continuous while admitting the representation
\begin{eqnarray}\label{rep1}
P_\O=(I+N^r_\O)^{-1}\;\mbox{ for some }\;r>0.
\end{eqnarray}
\end{Proposition}

Now we have all the ingredients allowing us to derive a precise second-order relationship between the directional derivatives of the projection operators $P_\Th$ and $P_\Gamma$ under the standing assumptions made. This result does not impose any other assumptions on $g$ and $\Th$ and fully eliminates the $\Th$-convexity of $g$ imposed in \cite{MOR}.

\begin{Theorem}[\bf directional differentiability of projections to conic constraints]\label{prop2} Let the standing assumptions be satisfied at some $\oy\in\Gamma$, and put $\ou=\oy $. Then there is a neighborhood $\mathcal{U}$ of $\bar{u}$ such that the single-valued projection operator $P_{\Gamma}$ onto $\Gamma$ is directionally differentiable at each $u\in\mathcal{U}$ in every direction $h\in\mathbb{R}^{m}$ and its directional derivative is calculated by $P'_{\Gamma}(u;h)=v_{1}$, where $v_1$ is the first component of the unique solution $v=(v_{1},v_{2})\in\mathbb{R}^{m}\times\mathbb{R}^{l}$ to the system of equations
\begin{equation}\label{eq-105}
\begin{array}{ll}
h=\big(I+\sum\limits^{l}_{i=1}\nu_{i}\nabla^{2}g_{i}(y)\big)v_{1}+(\nabla g(y))^Tv_{2},\\
0=\nabla g(y)v_{1}-P^{\prime}_{\Theta}\big(g(y)+\nu;\nabla g(y)v_{1}+v_{2}\big)
\end{array}
\end{equation}
with $y=P_{\Gamma}(u)$ and $\nu=(\nu_1,\ldots,\nu_l)\in\mathbb{R}^{l}$ being the unique Lagrange multiplier corresponding to the pair $(u,y)$ in the KKT system \eqref{eq-101} with $s=0$.
\end{Theorem}
\proof
Observe that the local single-valuedness and Lipschitz continuity of the projection operator $P_\Gamma$ follows directly from Proposition~\ref{thm1} due to the prox-regularity of $\Gamma$ at $\oy$.  To proceed further, define the mapping $\Phi:\mathbb{R}^{m}\times\mathbb{R}^{m}\times\mathbb{R}^{m}\times\mathbb{R}^{l}\to\mathbb{R}^{m}\times\mathbb{R}^{m}\times\mathbb{R}^{m}
\times\mathbb{R}^{l}$ by
\begin{eqnarray}\label{Phi}
\Phi(w,z,y,\nu):=\left[\begin{array}{ll}
w\\
z\\
y-w+\big(\nabla g(y)\big)^{T}\nu\\
g(y)+z-P_{\Theta}\big(g(y)+z+\nu\big)
\end{array}\right],
\end{eqnarray}
which is  single-valued and Lipschitz continuous around $(\ou,0,\oy,\bar \nu)$.
Using \eqref{Phi} and the definition of ${\cal T}$ in \eqref{eq-99}, we clearly get
\begin{eqnarray}\label{eq-104}
(y,\nu)\in{\cal T}(u,s)\Longleftrightarrow\Phi(w,z,y,\nu)=\left[\begin{array}{ll}
u\\
s\\
0\\
0
\end{array}\right].
\end{eqnarray}
It follows from Proposition~\ref{prop1}(ii) that there exist a neighborhood $\mathcal{O}$ of $(\bar{u},0,\bar{y},\bar{\nu})$ and a single-valued locally Lipschitzian mapping $\varrho$ such that $\varrho(\bar{u},0,0,0)=(\bar{u},0,\bar{y},\bar{\nu})$ and
\[
\varrho(\cdot)=\Phi^{-1}(\cdot)\cap\mathcal{O}
\]
on a neighborhood of $(\bar{u},0,0,0)$. Furthermore, the inverse mapping theorem by Kummer \cite{KK,Ku} tells us that $\varrho$ is directionally differentiable on a neighborhood of $(\bar{u},0,0,0)$ and its directional derivative on this neighborhood satisfies the relationship
\begin{eqnarray}\label{Phi1}
\varrho^{\prime}\big(\Phi(w,z,y,\nu);(h,0,0,0)\big)=\left[\begin{array}{ll}
h\\
0\\
v_{1}\\
v_{2}
\end{array}\right]\;\mbox{ with }\;\left[\begin{array}{ll}
h\\
0\\
0\\
0
\end{array}\right]=\Phi'\big((w,z,y,\nu);(h,0,v_{1},v_{2})\big).
\end{eqnarray}

To justify now the claimed representation of $P'_\Gamma(u;h)$, pick any $h\in\R^m$ and $u$ near $\ou$ and then find $(y,\nu)$ sufficiently close to $(\oy,\bar\nu)$ such that $u\in y+\Hat N_\Gamma(y)$ with $y\in\Gamma$ and $\Hat N_\Gamma(y)=N_\Gamma(y)$ due to the aforementioned prox-regularity of $\Gamma$; see \cite{RW}. Since $\ou=\oy$, we can choose $y\in\Gamma$ so that $\|y-u\|$ is small enough, which yields $u\in y+N^r_\Gamma(y)$ for the truncated normal cone in Proposition~\ref{thm1}. Employing representation \eqref{rep1} ensures that $y\in(I+N^r_\Gamma)^{-1}(u)=P_{\Gamma}(u)$. This allows us to combine the relationships in \eqref{eq-105}, \eqref{eq-104}, and \eqref{Phi1} and to arrive in this way at the conclusion of the theorem by using standard calculus rules of calculating the directional derivative of $\Phi$ in \eqref{Phi}.\endproof

\section{Projection Derivation Condition}\sce

In this section we introduce and comprehensively discuss a new condition on the underlying convex cone $\Theta$ formulated in terms of its projection operator $P_\Th$. This condition, together with the result of Theorem~\ref{prop2}, plays a crucial role in the precise calculation of the projection $P_\Gamma$ to the conic constraint set $\Gamma$ via the initial data of \eqref{eq-2} and then in the subsequent results of this paper.

The aforementioned property can be formulated for general sets in finite or infinite dimensions while we investigate and apply it below only for the set $\Th\subset\R^l$ under consideration in \eqref{eq-2}.

Given vectors $\oz\in\Th$ and $b\in\R^l$, define the {\em critical cone} to $\Th$ at $\oz$ with respect to $b$ by
\begin{eqnarray}\label{cc}
\mathcal{K}(\oz,b):=T_{\Theta}(\oz)\cap\{b\}^{\perp}.
\end{eqnarray}
\begin{Definition}[\bf projection derivation condition]\label{pdc} The set $\Th$ satisfies the {\sc projection derivation condition} $($PDC$)$ at the point $\oz\in\Th$ if we have
\begin{eqnarray}\label{pdc1}
P^{\prime}_{\Theta}(\oz+b;h)=P_{\mathcal{K}(\oz,b)}(h)\;\mbox{ for all }\;b\in N_{\Theta}(\oz)\;\mbox{ and }\;h\in\R^l.
\end{eqnarray}
\end{Definition}

Let us discuss the class of convex sets $\Th$ satisfying the new condition from Definition~\ref{pdc}. It follows from \cite{Rob} the PDC \eqref{pdc1} holds at each $\bar z\in\Theta$ when $\Theta$ is a convex polyhedron. In fact it also holds for a significantly broader collection of sets satisfying the so-called ``extended polyhedricity condition" from \cite[Definition~3.52]{BSbook}. To recall this definition, for the fixed vectors $\oz\in\Th$ and $b\in N_{\Theta}(\bar z)$ we define the {\em second-order critical set}
\begin{eqnarray}\label{K2}
{\cal K}^2(\bar z,b):=\big\{h\in{\cal K}(\bar z,b)\big|\;0\in T^2_\Theta(\bar z,h)\big\}
\end{eqnarray}
and say that $\Th$ satisfies the {\em extended polyhedricity condition} at $\bar z$ if for any $b\in N_{\Theta}(\bar z)$ the second-order critical set \eqref{K2} is a dense subset of the critical cone ${\cal K}(\bar z,b)$.

The next proposition reveals the relationship between the projection derivation condition and the extended polyhedricity condition defined above.

\begin{Proposition}[\bf extended polyhedricity implies PDC]\label{poly-pdc} Let $\Theta$ be a closed convex set with $\bar z\in\Theta$, and let $\Theta$ be cone reducible at $\bar z$. Then the validity of the extended polyhedricity condition for $\Th$ at $\bar z$ implies that $\Theta$ satisfies the projection derivation condition at this point.
\end{Proposition}
\begin{proof} As already mentioned, the cone reducibility ensures by \cite[Theorem~7.2]{BCS98} that the projection operator $P_\Theta$ is directionally differentiable at $\bar z$. Moreover, it is proved therein (see also \cite{BSbook,sh} for further details) that for any $u\in\R^l$ with
$\bar z=P_{\Theta}(u)$ we have the representation
\begin{eqnarray}\label{eq:derivative-projection}
P'_{\Theta}(u;h)=\argmin\big\{\|d-h\|^2-\sigma\big(u-\bar z;T^{2}_{\Theta}(\bar z,d)\big)\big|\;d\in{\cal K}(\bar z,u-\bar z)\big\},
\end{eqnarray}
where $\sigma(\cdot;\O):=\sup_{w\in\O}\langle\cdot,w\rangle$ stands for the support function of the set in question. Picking
an arbitrary vector $b\in N_{\Theta}(\bar z)$ and denoting $u:=\bar z+b$, we deduce from the well-known equivalence
\begin{equation}\label{eq:normal-proyection}
b\in N_{\Theta}(\bar z)\Longleftrightarrow P_{\Theta}(\bar z+b)=\bar z
\end{equation}
that $\bar z=P_{\Theta}(u)$. We claim that $P^{\prime}_{\Theta}(u;h)=P_{\mathcal{K}(\bar z,u-\bar z)}(h)$ for all $h\in\mathbb{R}^{l}$ whenever $\Theta$ satisfies the extended polyhedricity condition at $\bar z$.

Indeed, for an arbitrary element $d\in{\cal K}(\bar z,b)$ it is easy to see that $\sigma(b;T^{2}_{\Theta}(\bar z,d))\le 0$. Hence
$$
\min\big\{\|d-h\|^2-\sigma\big(u-\bar z;T^{2}_{\Theta}(\bar z,d)\big)\big|\;d\in{\cal K}(\bar z,u-\bar z)\big\}\ge\min\big\{\|d-h\|^2\big|\;d\in{\cal K}(\bar z,u-\bar z)\big\}.
$$
Moreover, since $ {\cal K}^2(\bar z,u-\bar z)\subset{\cal K}(\bar z,u-\bar z)$ and $\sigma(b;T^{2}_{\Theta}(\bar z,d))=0$ for any $d\in{\cal K}^2(\bar z,b)$, we get
\begin{eqnarray*}
&&\min\big\{\|d-h\|^2-\sigma\big(u-\bar z;T^{2}_{\Theta}(\bar z,d)\big)\big|\;d\in{\cal K}(\bar z,u-\bar z)\big\}\\
&&\qquad\quad\le\min\big\{\|d-h\|^2-\sigma(u-\bar z;T^{2}_{\Theta}(\bar z,d)\big)\big|\;d\in{\cal K}^2(\bar z,u-\bar z)\big\}\\
&&\qquad\quad=\min\big\{\|d-h\|^2\big|\;d\in{\cal K}^2(\bar z,u-\bar z)\big\}.
\end{eqnarray*}
This allows us to arrive at the equality
$$
\min\big\{\|d-h\|^2-\sigma\big(u-\bar z;T^{2}_{\Theta}(\bar z,d)\big)\big|\;d\in{\cal K}(\bar z,u-\bar z)\big\}=\min\big\{\|d-h\|^2\big|\;d\in{\cal K}(\bar z,u-\bar z)\big\}
$$
provided that ${\cal K}^2(\bar z,u-\bar z)$ is a dense subset of ${\cal K}(\bar z,u-\bar z)$, which is a consequence of the extended polyhedricity condition. Thus our claim follows from formula \eqref{eq:derivative-projection}.

Since $b=u-\bar z$, it follows from the above claim that
\[
P^{\prime}_{\Theta}(\bar z+b; h)=P^{\prime}_{\Theta}(u;h)=P_{\mathcal{K}(\bar z,u -\bar z)}(h)=P_{\mathcal{K}(\bar z,b)}(h)
\]
under the extended polyhedricity condition for $\Th$ at $\bar z$. Remembering that $b\in N_{\Theta}(\bar z)$ was chosen arbitrarily allows us to conclude that the PDC holds for $\Th$ at $\bar z$ and thus to complete the proof.
\end{proof}

The obtained proposition shows that the PDC property holds, in particular, for {\em polyhedric} convex sets (see, e.g., \cite[Definition~3.51]{BSbook}), which constitute a broader class that the standard convex polyhedra in finite dimensions. The next example describes a heavily nonpolyhedral situation when we do not have even polyhedricity but the PDC property holds.

\begin{Example}[\bf PDC for nonpolyhedric sets]\label{pdc-n} {\rm Consider the closed and convex cone
\begin{eqnarray}\label{nonpol}
\Theta:=\big\{z\in\mathbb{R}^{3}\big|\;z=tq\;\mbox{ with }\;t\ge 0,\;q\in\Xi\big\}
\end{eqnarray}
generated by the nonconvex three-dimensional set
$$
\Xi=\big\{z=(z_{1},z_{2},z_{3})\big|\;z_{1}=1\;\mbox{ and }\;z_{2}^{4}\le z_{3}\le 1\big\},
$$
which is depicted on Figure~\ref{fig:Xi-example} together with normals to $\Th$ at the reference point.

\begin{figure}[h]
\begin{center}
\includegraphics[scale=0.4]{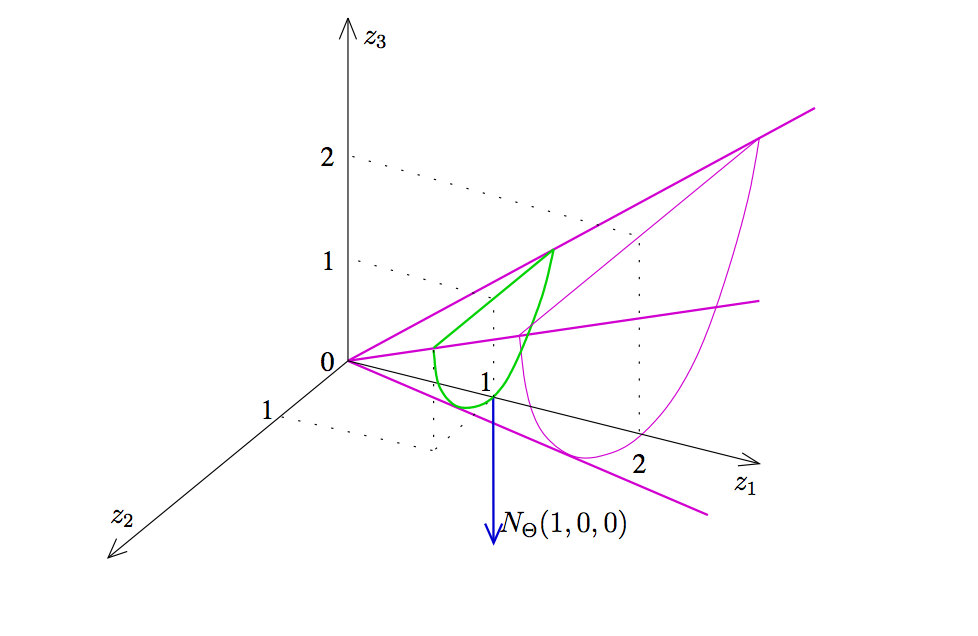}
\caption{Cone $\Theta$ from \eqref{nonpol}
}\label{fig:Xi-example}
\end{center}
\end{figure}

It is not hard to check that the set \eqref{nonpol} is nonpolyhedric at $\bar z=(1,0,0)$ and thus also nonpolyhedral. To show that $\Th$ has the PDC property at $\oz$, it suffices to check by Proposition~\ref{poly-pdc} that $\Theta$ satisfies the extended polyhedricity condition at this point. Since $\Theta$ is described around $\bar y$ by
\begin{eqnarray}\label{ph}
\varphi(z):=z_{2}^{4}/z_{1}^{3}-z_{3}\le 0,
\end{eqnarray}
it follows from \cite[Proposition~3.30]{BSbook} that for every $b\in N_{\Theta}(\bar z)$ and every $h\in{\cal K}(\bar z,b)$ the second-order tangent set $T^{2}_{\Theta}(\bar z,h)$ from \eqref{2tan} is given by
\begin{eqnarray}\label{T2}
T^2_\Theta(\bar z,h)=\big\{w\in\R^3\big|\;\varphi''(\bar z;h,w)\le 0\big\}
\end{eqnarray}
via the {\em parabolic second order directional derivative} of $\varphi$ defined by
\begin{eqnarray}\label{parab}
\varphi''(z;h,w):=\lim_{t\dn 0}\frac{\varphi(z+th+\half t^2w)-\varphi(z)-t\varphi'(z;h)}{\half t^2}.
\end{eqnarray}
Due to the twice continuous differentiability of $\ph$ from \eqref{ph} around $\oz$, we easily get
\begin{equation*}
\varphi''(z;h,w)=\nabla\varphi(z)w+\nabla^2\varphi(z)(h,h)
\end{equation*}
\begin{equation*}
=-3\frac{z_{2}^{4}}{z_{1}^{4}}w_{1}+4\frac{z_{2}^{3}}{z_{1}^{3}}w_{2}-w_{3}+12\left(\frac{z_{2}^{4}}{z_{1}^{5}}h_{1}^{2}-2\frac{z_{2}^{3}}{z_{1}^{4}} h_{1}h_{2}+\frac{z_{2}^{2}}{z_{1}^{3}}h_{2}^{2}\right).
\end{equation*}
Thus $N_{\Theta}(\bar z)=\{0\}\times\{0\}\times\mathbb{R}_{-}$ at $\bar z=(1,0,0)$ and then
\begin{eqnarray*}
{\cal K}(\bar z,b)=\left\{\begin{array}{ll}
\mathbb{R}^{2}\times\mathbb{R}_{+}&\mbox{if }\;b_{3}=0,\\
\mathbb{R}^{2}\times \{0\}&\mbox{if }\;b_{3}<0.
\end{array}\right.
\end{eqnarray*}
In both cases  formula \eqref{T2} leads us to the representation
\begin{eqnarray*}
T^2_\Th(\oz,h)=\big\{w\in\R^3\big|\;w_{3}\ge 0\big\}\;\mbox{ for any }\;h\in{\cal K}(\bar z,b),
\end{eqnarray*}
and so $0\in T^2_\Th(\oz,h)$. It shows that ${\cal K}^2(\bar z,b)={\cal K}(\bar z,b)$ for the second-order critical set in \eqref{K2}, and thus $\Theta$ satisfies the extended polyhedricity condition at $\bar z$.

Observe further that the latter property can be lost for the set $\Th$ from \eqref{nonpol} at other points, where the curvature of $\Th$ is larger.
Indeed, at $\bar z=(1,1,1)$ the same computations as above lead us to $N_{\Theta}(\bar z)=\mathbb{R}_{+}(-3,4,-1)$, and thus for any $b\ne 0$ we have
$$
{\cal K}(\bar z,b)=\big\{h\in\R^3\big|\;-3h_{1}+4h_{2}-h_{3}=0\big\}.
$$
It gives us for any $h\in{\cal K}(\bar z,b)$ the representation
$$
T^2_\Theta(\bar z,h)=\big\{ w\in\R^3\big|\;-3w_{1}+4w_{2}-w_{3}+12(h_{1}-h_{2})^{2}\le 0\big\}.
$$
Hence $0\in T^2_\Theta(\bar z,h)$ if and only if $h_{1}=h_{2}$. This shows that ${\cal K}^2(\bar z,b)$ is not a dense subset of
the critical cone ${\cal K}(\bar z,b)$, and so the extended polyhedricity property of $\Th$ is violated at this point.}\hfil$\triangle$
\end{Example}

Since the PDC property is local, it may hold in many important situations that have nothing to do with polyhedricity. The next proposition reveals one of them, which is used what follows.

\begin{Proposition}\label{Lemma:ext-polyhedricity-vertex}
Any closed and convex cone $\Th\subset\R^l$ satisfies PDC  at the vertex $\oz=0\in\Th$.
\end{Proposition}
\begin{proof}
By Proposition~\ref{poly-pdc} it suffices to check that $\Th$ satisfies the extended polyhedrality condition at $\oz=0$. Observe to this end that $T_{\Theta}(0)=\Theta$ and that $\Th$ clearly contains the critical cone ${\cal K}(0,b)$ for any $b\in N_{\Theta}(0)=\Theta^*$. Thus it remains to note that for any $h\in\Theta$ we have $T^{2}_{\Theta}(0,h)=T_{\Theta}(h)$. This yields that $0\in T^{2}_{\Theta}(0,h)$, and so the condition ${\cal K}^2(0,h)={\cal K}(0,h)$ holds for any $h\in\Theta$.
\end{proof}

\section{Calculating Graphical Derivatives}\sce

First we present the following result on calculating the graphical derivative of the projection $P_\Gamma$, which is of its own interest while being used in establishing the main results given below.

\begin{Lemma}[\bf graphical derivative of projections to conic constraints]\label{prop3} Let $\bar{y}\in P_{\Gamma}(\bar{u})$, and let $\bar{\nu}$ be the corresponding unique multiplier satisfying \eqref{eq-99}. Then
\begin{eqnarray}\label{eq-200}
\begin{array}{ll}
DP_{\Gamma}(\bar{u},\bar{y})(h)=&\big\{v_{1}\in\mathbb{R}^{m}\big|\;\exists~v_{2}\in\mathbb{R}^{l}\;\mbox{ such that }\\
&h=\left(I+\sum\limits^{l}_{i=1}\bar{\lambda}_{i}\nabla^{2}g_{i}(\bar{y})\right)v_{1}+\big(\nabla g(\bar{y})\big)^{T}v_{2},\\
&0=\nabla g(\bar{y})v_{1}-P^{\prime}_{\Theta}\big(g(\bar{y})+\bar{\nu};\nabla g(\bar{y})v_{1}+v_{2}\big)\big\}
\end{array}
\end{eqnarray}
under the standing assumptions made. Assuming in addition that PDC is satisfied at $\oz:=g(\bar{y})$ and denoting $\Bar{\cal K}:=\mathcal{K}(g(\bar{y}),\bar{\nu})$, we have
\begin{eqnarray}\label{eq-201}
DP_{\Gamma}(\bar{u},\bar{y})(h)=\left\{v\left|\;h\in\left(I+\sum\limits^{l}_{i=1}\bar{\nu}_{i}\nabla^{2}g_{i}(\bar{y})\right)v+\big(\nabla g(\bar{y})\big)^{T}N_{\Bar{\cal K}}\big(\nabla g(\bar{y})v\big)\right.\right\}.
\end{eqnarray}
\end{Lemma}
\proof Formula \eqref{eq-200} follows from Theorem~\ref{prop2} since the graphical derivate reduces to the directional one when the latter exists.
To verify \eqref{eq-201} under the imposed PDC assumption, observe that the second relationship on the right-hand side of (\ref{eq-200}) amounts in this case to saying that
\begin{eqnarray}\label{eq-202}
P_{\Bar{\cal K}}\big(\nabla g(\bar{y})v_{1}+v_{2}\big)=\nabla g(\bar{y})v_{1}.
\end{eqnarray}
Since the cone $\Bar{\cal K}$ is surely convex due to the convexity of $\Theta$, equality (\ref{eq-202}) is equivalent to the inclusion $v_{2}\in N_{\Bar{\cal K}}(\nabla g(\bar{y})v_{1})$. In this way we arrive at \eqref{eq-201}.
\endproof

The next major result provides a precise second-order formula for calculating the graphical derivative of the regular normal cone mapping \eqref{eq-1} to the conic constraint set $\Gamma$ from \eqref{eq-2}.

\begin{Theorem}[\bf graphical derivative of the normal cone mapping]\label{thm2} Let $\bar{w}\in\widehat{N}_{\Gamma}(\bar{y})$, and let $\bar{\nu}$ be the unique multiplier satisfying the KKT system in \eqref{eq-99} with $\bar{w}=\bar{u}-\bar{y}$. In addition to the standing assumptions made, suppose that the cone $\Th$ in \eqref{eq-2} satisfies PDC at $\oz=g(\oy)$. Then for any $v\in\mathbb{R}^{m}$ we have the representation
\begin{eqnarray}\label{eq-203}
D\widehat{N}_{\Gamma}(\bar{y},\bar{w})(v)=\left(\sum\limits^{l}_{i=1}\bar{\nu}_{i}\nabla^{2}g_{i}(\bar{y})\right)v+\nabla g(\bar{y})^{T}N_{\Bar{\cal K}}\big(\nabla g(\bar{y})v\big),
\end{eqnarray}
where the cone $\Bar{\cal K}$ is defined in Lemma~{\rm\ref{prop3}}.
\end{Theorem}
\proof As mentioned above, the conic constraint set $\Gamma$ is prox-regular at $\bar{y}$. Thus $N_{\Gamma}(y)=\widehat{N}_{\Gamma}(y)$ for all $y\in\Gamma$ sufficiently close to $\oy$ and the result of Proposition~\ref{thm1} can be applied. It follows therefore that there exists a neighborhood $\mathcal{U}$ of $\bar{y}$ such that for all $u\in\mathcal{U}$ we have the equivalence
\[
y=P_{\Gamma}(u)\Longleftrightarrow u\in y+N_{\Gamma}(y).
\]
Consider now a neighborhood $\mathcal{V}$ of $\bar{y}$ and an $\eps > 0$ such that $y+w\in \mathcal{U}$ for all $y\in\mathcal{V}$ and $w\in\eps\B$. It follows furthermore that
\begin{eqnarray}\label{eq-204}
w\in\widehat{N}_{\Gamma}(y)\;\mbox{ if and only if }\;\left[\begin{array}{ll}
y+w\\
y
\end{array}\right]\in\gph P_{\Gamma}
\end{eqnarray}
provided that $y\in\mathcal{V}$ and $w\in\eps\B$. Given $\bar{w}\in \hat{N}_\Gamma(\bar{y})$, we can find a positive number $\vartheta$ such that $\vartheta\bar{w}\in \frac{\eps}{2}\B$, and so the equivalence (\ref{eq-204}) holds for all $y\in\Gamma$ close to $\bar{y}$ and $w$ close to $\vartheta\bar{w}$. Combining (\ref{eq-204}) with formula (\ref{eq-201}) from Lemma~\ref{prop3} and the chain rule from \cite[Exercise 6.7]{RW} yields
\[
D\widehat{N}_{\Gamma}(\bar{y},\vartheta\bar{w})(v)=\left(\sum\limits^{l}_{i=1}\nu_{i}\nabla^{2}g_{i}(\bar{y})\right)v+\nabla g(\bar y)^{T}N_{\Bar{\cal K}}\big(\nabla g(\bar{y})v\big),
\]
where $\nu\in\R^l$ is the unique multiplier satisfying the conditions
\[
\vartheta\bar{w}=\nabla g(\bar{y})^{T}\nu,\quad\nu\in N_{\Theta}\big(g(\bar{y})\big).
\]
It remains to denote $\bar{\nu}:=\frac{\nu}{\vartheta}$ and recall the easily verifiable equivalence (see \cite[Lemma~1(i)]{Kr})
\begin{eqnarray*}
(h,s)\in T_{{\rm gph}\,\Xi}(a, b)\Longleftrightarrow(h,\vartheta s)\in T_{{\rm gph}\,\Xi}(a,\vartheta b),\quad\vartheta>0,
\end{eqnarray*}
which holds for any cone-valued mapping $\Xi\colon\R^n\tto\R^n$ with $(a,b)\in\gph\Xi$. This together with definition \eqref{der} of the graphical derivative gives us \eqref{eq-203} and completes the proof of the theorem.
\endproof

Now we are ready to present the final result of this section giving us an upper estimate of the graphical derivative of the solution map $S$ in \eqref{eq:solution-map} under the assumptions above and then a precise representation under an additional surjectivity assumption. It is convenient to formulate this result via the Lagrangian function associated with GE \eqref{eq-1} by
\begin{eqnarray*}
{\cal L}(x,y,\lambda):=f(x,y)+\nabla g(y)^{T}\lambda,\quad\lm\in\R^l.
\end{eqnarray*}
\begin{Theorem}[\bf graphical derivative of the solution map]\label{thm3}
Let $(\bar x,\bar y)\in\gph S$, $\bar w:=-f(\bar x,\bar y)$, and $\bar\lambda\in N_{\Theta}(g(\bar y))$ be the unique Lagrange multiplier satisfying the equation
$$
{\cal L}(\bar x,\bar y,\bar\lambda)=0.
$$
Suppose that all the assumptions of Theorem~{\rm\ref{thm2}} are fulfilled and that $f$ is a smooth vector function around $(\ox,\oy)\in\gph S$. Then for any $u\in\rr^{n}$ we have the inclusion
\begin{equation}\label{eq:inclusion-solution-map}
DS(\bar x,\bar y)(u)\subset\big\{v\in\rr^{m}\big|\;0\in\nabla_{x}f(\bar x,\bar y)u+\nabla_{y}{\cal L}(\bar x,\bar y,\bar\lambda)v+\big(\nabla g(\bar y)\big)^{T}N_{\Hat{\cal K}}\big(\nabla g(\bar y)v\big)\big\}
\end{equation}
with the notation $\Hat{\cal K}:={\cal K}(g(\bar y),\bar \lambda)$. Furthermore, inclusion \eqref{eq:inclusion-solution-map} becomes an equality provided that partial Jacobian $\nabla_{x}f(\bar x,\bar y)$ is surjective.
\end{Theorem}
\proof
We obviously have the representation
$$
\gph S=\big\{(x,y)\in\R^n\times\R^m\big|\;\big(y,-f(x,y)\big)\in\gph\Hat N_{\Gamma}\big\}.
$$
It follows from the tangent cone calculus rule of\cite[Theorem~6.31]{RW} that we have the inclusion
\begin{eqnarray}\label{eq:tangent-cone}
T_{{\rm gph}\,S}(\bar x,\bar y)\subset\big\{(u,v)\in\R^n\times\R^m\big|\;-\nabla_{x}f(\bar x,\bar y)u-\nabla_{y}f(\bar x,\bar y)v\in D\Hat N_\Gamma\big(\bar y,f(\bar x,\bar y)\big)(v)\big\}.
\end{eqnarray}
Moreover, \eqref{eq:tangent-cone} holds as equality provided that the matrix $\nabla_{x}f(\bar x,\bar y) $ is surjective; see, e.g., \cite[Exercise~6.7]{RW}. Combining inclusion \eqref{eq:tangent-cone} with Theorem~\ref{thm2} gives us the upper estimate \eqref{eq:inclusion-solution-map} while the additional surjectivity assumption ensures the equality therein and thus completes the proof.
\endproof

\section{Application to Isolated Calmness}\sce

In this section we develop an application of the graphical derivative evaluations obtained in Theorem~\ref{thm3} to derive sufficient as well as necessary and sufficient conditions for the so-called isolated calmness of solution map $S$ from \eqref{eq:solution-map}, which is a useful local Lipschitzian stability property recognized in variational analysis and optimization; see, e.g., \cite{DR} and the references therein.

\begin{Definition}[\bf isolated calmness]\label{il-def} We say that a set-valued mapping $F:\R^{d}\rightrightarrows\mathbb{R}^{s}$ has the {\sc isolated calmness property} at $(\bar x,\bar y)\in\gph F$ if there exist neighborhoods $U$ of $\bar x$ and $V$ of $\bar y$ as well as a positive constant $\ell>0$ such that
\begin{eqnarray}\label{calm}
F(x)\cap V\subset\{\bar y\}+\ell\|x-\bar x\|\B\;\mbox{ for all }\;x\in U.
\end{eqnarray}
\end{Definition}

This property can be viewed as a local single-valued restriction at the nominal point $\oy$ of the {\em calmness} notion for set-valued mappings \cite{RW}, which in turn is an image localization of Robinson's {\em upper Lipschitz} property introduced in \cite{rob} for stability analysis of generalized equations. Note that the isolated calmness is called ``local upper Lipschitz" property in \cite{l}. Furthermore, it is easy to show (see, e.g., \cite[Theorem~3I.2]{DR}) that property \eqref{calm} for $F$ is equivalent with the so-called ``strong metric subregularity" of the inverse mapping $F^{-1}$. It is worth mentioning that the latter property has been recently applied  in \cite{dmn} to the study of {\em tilt stability} in optimization, which is a particular case of full stability used in Section~3 to derive formulas for the directional derivatives of projections to nonconvex conic constraints. As we have seen, these formulas have been much employed in deriving the main graphical derivative results in Section~5. Some characterizations of the isolated calmness and strong metric subregularity properties for parametric variational inequalities over polyhedral convex sets can be found in \cite{DR} and \cite{Kr}.

Our application of the obtained graphical derivative calculations to isolated calmness is based on the following graphical derivative characterization of this property for general multifunctions between finite-dimensional spaces whose necessity part was obtained in \cite[Proposition~2.1]{kr} while sufficiency was proved later in \cite[Proposition~4.1]{l}.

\begin{Lemma}[\bf graphical derivative criterion for isolated calmness]\label{ic-char} Let $F\colon\R^d\tto\R^s$, and let $(\ox,\oy)\in\gph F$. Then $F$ has the isolated calmness property at $(\ox,\oy)$ if and only if $DF(\ox,\oy)(0)=\{0\}$.
\end{Lemma}

The last result of this paper incorporates the graphical derivative evaluation for the solution map \eqref{eq:solution-map} into the isolated calmness criterion of Lemma~\ref{ic-char}. In this way we arrive at efficient conditions for the isolated calmness property of solutions to GE \eqref{eq-1} in terms of its initial data.

\begin{Theorem}[\bf isolated calmness for parameterized equilibria with conic constraints]\label{thm4} In the setting of Theorem~{\rm\ref{thm3}}, assume that the adjoint generalized equation
\begin{equation}\label{eq:GE-isolated-calm}
0\in\nabla_{y}{\cal L}(\bar x,\bar y,\bar\lambda)v+\big(\nabla g(\bar y)\big)^{\top} N_{\hat{\cal K}}\big(\nabla g(\bar y)v\big)
\end{equation}
has only the trivial solution $v=0$. Then the solution map $S$ from \eqref{eq:solution-map} has the isolated calmness property at $(\bar x, \bar y)$. If in addition the partial Jacobian $\nabla_{x} f(\bar x,\bar y)$ is surjective, then the above condition is also necessary for $S$ to have the isolated calmness property at $(\bar x,\bar y)$.
\end{Theorem}
\proof This is a direct combination of Lemma~\ref{ic-char} and Theorem~\ref{thm3}.
\endproof

Finally in this section, we illustrate the usage of Theorem~\ref{thm4} in the case of nonpolyhedral conic constraints in \eqref{eq-1} with $\Theta$ being the second-order cone \eqref{ic} in $\R^{3}$.

\begin{Example}[\bf isolated calmness for equilibrium problems with second-order cone constraints] {\rm Consider the generalized equation \eqref{eq-1} with $x,y\in\rr^{3}$, $f(x,y)=x$,
$$
\Theta:={Q}^3=\left\{(\th_1,\th_{2},\th_3)\in\R^3\big|\;\th_3\ge\|(\th_{1},\th_2)\|=\sqrt{\th_{1}^{2}+\th_{2}^{2} }\right\},
$$
and $\Gamma:=g^{-1}(\Theta)$ with $g(y):=(y_{1},y_{2},y_{3}+0.2(y_{1}^{2}+y_{2}^{2}))$, i.e.,
$$
\Gamma=\left\{(y_{1},y_{2},y_{3})\in\R^3\big|\;\big(y_{1},y_{2},y_{3}+0.2(y_{1}^{2}+y_{2}^{2})\big)\in {Q}^3\right\}.
$$
This set $\Gamma$ is clearly nonconvex.

Note that the corresponding generalized equation \eqref{eq-1} amounts to the stationary condition for the parametric optimization problem given by:
\[
\mbox{minimize }\;\langle x,y\rangle\quad\mbox{ subject to }\quad y\in\Gamma.
\]
Consider the pair $(\bar x,\bar y)$ with $\bar x=(-1,0,1)$ and $\bar y=(0,0,0)$, which belongs to the graph of the solution map $S$ of this GE.
Since $\nabla g(\bar y)=I$, it follows that $\bar \lambda=-\bar x$ and the vector $\bar y$ trivially satisfies the nondegeneracy condition {\bf(A2)}. As we pointed out in Section~2, ${Q}^3$ is cone reducible, and therefore its metric projection is directionally differentiable everywhere on $\R^3$. Thus our standing assumptions {\bf (A1)} and {\bf (A3)} are also satisfied in this setting. Furthermore, Proposition~\ref{Lemma:ext-polyhedricity-vertex} ensures that $\Theta=Q^3$ satisfies the PDC property at its vertex $\bar y=(0,0,0)$.

Let us now show that the solution map $S$ enjoys the isolated calmness property at $(\bar x,\bar y)$ via the verification of condition \eqref{eq:GE-isolated-calm} from Theorem~\ref{thm4}. Observe to this end that
$$
\nabla_{y}{\cal L}(\bar x,\bar y,\bar\lambda)v=\bar\lambda_{3}\nabla^{2}g_{3}(\bar y)v=\left(\begin{matrix}-0.4v_{1}\\-0.4v_{2}\\0\end{matrix}\right)\;\mbox{ and }\;\Hat{\cal K}={\cal K}(\bar y,\bar\lambda)={\cal K}^{3}\cap\bar\lambda^{\bot}=\R_{+}\left(\begin{matrix}1\\0\\1\end{matrix}\right).
$$
Since $N_{\hat{\cal K}}\big(\nabla g(\bar y)v\big)=\hat{\cal K}^*\cap v^\bot$ (due to $\nabla g(\bar y)=I$), condition \eqref{eq:GE-isolated-calm} amounts to the implication
$$
\left[\left(\begin{matrix}0.4v_{1}\\0.4v_{2}\\0\end{matrix}\right)\in\Hat{\cal K}^*,\;v\in\hat{\cal K},\;v_{1}^{2}+v_{2}^{2}=0\right]\Longrightarrow v=0.
$$
By the direct calculation we have $\Hat{\cal K}^*=\{a\in\R^3|\;a_{1}+a_{3}\le 0\}$, the so the above implication holds. This shows by Theorem~\ref{thm4} that the solution map $S$ in this example possesses the isolated calmness property at $(\bar x,\bar y)$. It is worth noting that $S$ does not have the (robust) Aubin/Lipschitz-like property around $(\bar x,\bar y)$ because its values are empty for all $x<0$.}
\end{Example}

\section{Concluding Remarks}\sce

This paper demonstrates that the recently developed techniques of second-order variational analysis and full stability in optimization allow us to derive calculus formulas for graphical derivatives of solution maps to parameterized equilibria with conic constraints in challenging cases of nonconvex constraint sets generated by nonpolyhedral cones. The new projection derivation condition plays a crucial role in obtaining verifiable results in this direction and their application to isolated calmness of solution maps. This condition is local and may be violated in many situations. In such cases we do not have for now an efficient technique for calculating graphical derivatives in our disposal. This could be an interesting goal for further research. On the other hand, the obtained formulas for calculating graphical derivatives of metric projections and solutions maps contain terms expressed via normals to the corresponding (convex) critical cone for the underlying set $\Th$ at the reference points. It would be appealing from both viewpoints of optimization/equilibrium theory and its applications to further evaluate these terms entirely via the initial data of remarkable constraint systems appearing in conic programming.\\[2ex]
{\bf Acknowledgements.} The authors are grateful to Fr\'ed\'eric Bonnans, Ebrahim Sarabi, Alex Shapiro, and Lionel Thibault for helpful discussions on various aspects of this paper.

\end{document}